\documentclass[3p,times]{elsarticle}

\journal{Elsevier}

\usepackage{amsmath}
\usepackage{mathrsfs} 
\newcommand{\domD}{\mathscr{D}}
\renewcommand{\pi}{\piup}
\renewcommand{\Re}{\operatorname{Re}}
\renewcommand{\Im}{\operatorname{Im}}
\DeclareMathOperator{\arsinh}{arsinh}
\DeclareMathOperator{\OO}{O}
\makeatletter
\newif\ifkp@upRm
\DeclareSymbolFont{Letters}{OML}{jkp}{m}{n}
\DeclareMathSymbol{\uppartial}{\mathord}{Letters}{128}
\makeatother
\DeclareMathOperator{\E}{e}
\DeclareMathOperator{\I}{i}

\newdefinition{definition}{Definition}[section]
\newtheorem{theorem}{Theorem}[section]
\newtheorem{lemma}[theorem]{Lemma}
\newtheorem{proposition}[theorem]{Proposition}

\newtheorem{remark}{Remark}[section]

\newproof{proof}{Proof}
\numberwithin{equation}{section}

\begin{document}

\begin{frontmatter}

\title{Improvement of conformal maps combined with the Sinc approximation for derivatives over infinite intervals~\tnoteref{mytitlenote}}
\tnotetext[mytitlenote]{This work was partially supported by JSPS
Grant-in-Aid for Scientific Research (C) JP23K03218.}

\author[HCU]{Tomoaki Okayama\corref{cor1}}
\cortext[cor1]{Corresponding author}
\affiliation[HCU]{organization={Hiroshima City University},
addressline={3-4-1, Ozuka-higashi, Asaminami-ku},
city={Hiroshima},
postcode={731-3194},
country={Japan}}
\ead{okayama@hiroshima-cu.ac.jp}


\author[secondauthor]{Yuito Kuwashita}
\affiliation[secondauthor]{organization={Arcadia Plus Inc.},
addressline={Sakura Dori Otsu KT Building 8F, 3-20-22, Marunouchi, Naka-ku},
city={Nagoya},
postcode={460-0002},
country={Japan}}

\author[HCU]{Ao Kondo}




\begin{abstract}
F. Stenger proposed
efficient approximation formulas for derivatives over infinite intervals.
These formulas were derived by combining
the Sinc approximation with appropriate conformal maps.
It has been demonstrated that these formulas can attain root-exponential convergence.
In this study, we enhance the convergence rate by improving
the conformal maps employed in those formulas.
We provide a theoretical error analysis and numerical experiments
that confirm the effectiveness of our new formulas.
\end{abstract}


\end{frontmatter}

\section{Introduction}
\label{sec:intro}

This study focuses on the approximation formulas
for derivatives based on the Sinc approximation
\begin{equation}
\label{eq:Sinc-approximation}
 F(x) \approx \sum_{k=-M}^N F(kh)S(k,h)(x),\quad x\in\mathbb{R},
\end{equation}
where $h$ is the mesh size,
$M$ and $N$ are truncation numbers,
and $S(k, h)$ is the so-called Sinc function defined by
\[
 S(k,h)(x)=
\begin{cases}
\dfrac{\sin[\pi(x - kh)/h]}{\pi(x - kh)/h} & (x\neq kh), \\
1 & (x = kh).
\end{cases}
\]
The Sinc approximation~\eqref{eq:Sinc-approximation}
is known to be efficient
(nearly optimal~\cite{stenger78:_optim_h,sugihara03:_near})
for analytic functions $F$
that satisfy the following two conditions:
(i) $F(x)$ is defined on the entire real axis $\mathbb{R}$, and
(ii) $|F(x)|$ decays exponentially as $x\to\pm\infty$.
When those conditions are not satisfied,
Stenger~\cite{stenger93:_numer,Stenger} proposed
to employ an appropriate conformal map
depending on the target interval $(a,b)$ and
the decay rate of the given function $f$.
He considered the following five typical cases:
\begin{enumerate}
 \item $(a, b)=(0,\infty)$ and $|f(t)|$ decays algebraically as $t\to\infty$,
 \item $(a, b)=(0,\infty)$ and $|f(t)|$ decays exponentially as $t\to\infty$,
 \item $(a, b)=(-\infty,\infty)$ and $|f(t)|$ decays algebraically as $t\to\pm\infty$,
 \item $(a, b)=(-\infty,\infty)$ and $|f(t)|$ decays algebraically as $t\to-\infty$ and exponentially as $t\to\infty$,
 \item the interval $(a,b)$ is finite.
\end{enumerate}
For each case, he presented a recommended conformal map
$\psi_i$ $(i=1,\,\ldots,\,5)$ as
\begin{align*}
 t &= \psi_1(x) = \E^x,\\
 t &= \psi_2(x) = \arsinh(\E^x),\\
 t &= \psi_3(x) = \sinh x,\\
 t &= \psi_4(x) = \sinh(\log(\arsinh(\E^x))),\\
 t &= \psi_5(x) = \frac{b - a}{2}\tanh\left(\frac{x}{2}\right)+\frac{b + a}{2}.
\end{align*}
Combination of the conformal map $\psi_i$ with
the Sinc approximation~\eqref{eq:Sinc-approximation} gives
an approximation for a function $f(t)$ as
\begin{equation}
\label{eq:psi-Sinc-approximation}
 f(t)\approx \sum_{k=-M}^N f(\psi_i(kh))S(k,h)(\psi_i^{-1}(t)),
\quad t\in (a,b).
\end{equation}
For a derivative of $f$, one may naturally consider
differentiating both sides of~\eqref{eq:psi-Sinc-approximation} as
\begin{equation}
 f'(t)\approx \sum_{k=-M}^N f(\psi_i(kh))S'(k,h)(\psi_i^{-1}(t))
\left\{\psi_i^{-1}(t)\right\}',
\quad t\in (a,b).
\end{equation}
However, $\{\psi_i^{-1}(t)\}'$ diverges at endpoints in some cases
(more precisely, in the cases $i=1,\, 2,\, 5$).
In such cases,
we cannot expect uniform approximation over the target interval $(a, b)$.
To address this issue,
in addition to $\psi_i$,
Stenger~\cite{stenger93:_numer} proposed to use
an appropriate function $g_i$ as
\begin{align}
 \frac{f(t)}{g_i(t)}
&\approx \sum_{k=-M}^N \frac{f(\psi_i(kh))}{g_i(\psi_i(kh))}
S(k,h)(\psi_i^{-1}(t)),\quad t\in(a,b),\nonumber
\intertext{which is equivalent to}
 f(t) &\approx \sum_{k=-M}^N \frac{f(\psi_i(kh))}{g_i(\psi_i(kh))}
g_i(t)S(k,h)(\psi_i^{-1}(t)),\quad t\in(a,b).
\label{eq:psi-g-Sinc-approximation}
\end{align}
The functions $g_i$ ($i=1,\,\ldots,\,5$) are given as
\begin{align*}
 g_1(t) &= \left(\frac{t}{1+t}\right)^m,\\
 g_2(t) &= \left(1 - \E^{-t}\right)^m,\\
 g_3(t) &= 1,\\
 g_4(t) &= 1,\\
 g_5(t) &= (t - a)^m(b - t)^m,
\end{align*}
which are chosen so that
$g_i(t)$ suppresses the divergence of $\{\psi_i^{-1}(t)\}^{(m)}$.
The concrete form of the approximation formula
for $m$-th derivative is derived by differentiating both sides
of~\eqref{eq:psi-g-Sinc-approximation} as
\begin{equation}
 f^{(m)}(t)
\approx \sum_{k=-M}^N \frac{f(\psi_i(kh))}{g_i(\psi_i(kh))}
\left(\frac{\mathrm{d}}{\mathrm{d}t}\right)^m
\left\{
g_i(t)S(k,h)(\psi_i^{-1}(t))
\right\},\quad t\in(a, b).
\label{eq:Stenger-formula}
\end{equation}
He also conducted a theoretical error analysis and
showed that this approximation formula yields a uniform
approximation over the target interval, and can attain
root-exponential convergence: $\OO(\exp(-c\sqrt{n}))$,
where $c$ is a positive constant and $n=\max\{M,N\}$.
Owing to such high efficiency, several authors
have utilized the formula to solve
differential equations~\cite{bialecki91:_sinc,lund92:_sinc_method_quadr_differ_equat,morlet95:_conver,saadatmandi07}.

This study aims to improve the convergence
rate of the formula in the cases $i=2$ and $i=4$.
The concept for the improvement is replacing the conformal maps;
for $i=2$, we replace $\psi_2(x)$ with
\begin{align*}
 \phi_2(x) &= \log(1 + \E^{x}),
\intertext{and for $i=4$, we replace $\psi_4(x)$ with}
 \phi_4(x) &= 2\sinh(\log(\log(1 + \E^x))).
\end{align*}
Such replacement of the conformal maps
has already been conducted
for function approximation~\cite{OkaShinKatsu,okayama21:_improv_sinc}
and integral approximation~\cite{OkayamaMachida,tomoaki21:_new},
and improvement of the convergence rate has been reported.
Furthermore, in the case of integral approximation,
it has been theoretically revealed that
$\phi_2$ is \emph{always} superior to $\psi_2$~\cite{okayama22:_theor}.
Considering these studies as motivation,
we propose a new approximation formula for $m$-th derivative as
\begin{equation}
  f^{(m)}(t)
\approx \sum_{k=-M}^N \frac{f(\phi_i(kh))}{g_i(\phi_i(kh))}
\left(\frac{\mathrm{d}}{\mathrm{d}t}\right)^m
\left\{
g_i(t)S(k,h)(\phi_i^{-1}(t))
\right\},\quad t\in(a, b),
\label{eq:new-formula}
\end{equation}
for $i=2$ and $i=4$.
We also conduct a theoretical error analysis claiming that
the improved formulas can attain $\OO(\exp(-c'\sqrt{n}))$,
where $c'$ is a constant that may be greater than $c$.

The remainder of this paper is organized as follows.
Convergence theorems for Stenger's formula
are summarized in Section~\ref{sec:Stengers-formula}.
Convergence theorems for the new approximation formula
are described in Section~\ref{sec:improve-converge}
along with their proofs.
Numerical examples are presented in Section~\ref{sec:numer}.
Finally,
the conclusions of this study are presented in Section~\ref{sec:conclusion}.

\section{Convergence theorems for Stenger's formula}
\label{sec:Stengers-formula}


In this section, we describe the convergence theorems of
Stenger's formula~\eqref{eq:Stenger-formula}
for $i=2$ and $i=4$.
In the original theorems,
the summation is not $\sum_{k=-M}^N$ but $\sum_{k=-N}^N$.
However, extension to $\sum_{k=-M}^N$ is relatively straightforward
as follows.
Here, let $\domD_d$ be a strip domain
defined by $\domD_d=\{\zeta\in\mathbb{C}:|\Im \zeta|< d\}$
for $d>0$.
Furthermore, let
$\domD_d^{-}=\{\zeta\in\domD_d :\Re\zeta< 0\}$
and
$\domD_d^{+}=\{\zeta\in\domD_d :\Re\zeta\geq 0\}$.

\begin{theorem}[Stenger~{\cite[Theorem~4.4.2 and Example~4.4.6]{stenger93:_numer}}]
\label{thm:Stenger-2}
Assume that $f$ is analytic in $\psi_2(\domD_d)$ with $0<d<\pi/2$,
and that there exist positive constants $K$, $\alpha$ and $\beta$
such that
\begin{equation}
 \left|\frac{f(z)}{g_2(z)}\right|
\leq K \left|\frac{z}{1+z}\right|^{\alpha}|\E^{-z}|^{\beta}
\label{eq:f-g_2-bound}
\end{equation}
holds for all $z\in\psi_2(\domD_d)$.
Let $\mu = \min\{\alpha,\beta\}$,
let $M$ and $N$ be defined as
\begin{equation}
M = \left\lceil\frac{\mu}{\alpha}n\right\rceil,
\quad
N = \left\lceil\frac{\mu}{\beta }n\right\rceil,
\label{eq:Def-MN}
\end{equation}
and let $h$ be defined as
\begin{equation}
h = \sqrt{\frac{\pi d}{\mu n}}.
\label{eq:Def-h}
\end{equation}
Then, there exists a constant $C$ independent of $n$ such that
\[
 \sup_{t\in (0,\infty)}\left|
 f^{(m)}(t)
- \sum_{k=-M}^N \frac{f(\psi_2(kh))}{g_2(\psi_2(kh))}
\left(\frac{\mathrm{d}}{\mathrm{d}t}\right)^m
\left\{
g_2(t)S(k,h)(\psi_2^{-1}(t))
\right\}
\right|
\leq C n^{(m+1)/2}\exp\left(-\sqrt{\pi d \mu n}\right).
\]
\end{theorem}
\begin{theorem}[Stenger~{\cite[Theorem~1.5.4]{Stenger}}]
\label{thm:Stenger-4}
Assume that $f$ is analytic in $\psi_4(\domD_d)$ with $0<d<\pi/2$,
and that there exist positive constants $K$, $\alpha$ and $\beta$
such that
\begin{align}
 |f(z)|&\leq K |\E^{-z}|^{2\beta}\label{leq:f-Dd-plus-original}
\intertext{holds for all $z\in\psi_4(\domD_d^{+})$, and}
 |f(z)|&\leq K \frac{1}{|z|^{\alpha}}\label{leq:f-Dd-minus-original}
\end{align}
holds for all $z\in\psi_4(\domD_d^{-})$.
Let $\mu = \min\{\alpha,\beta\}$,
let $M$ and $N$ be defined as~\eqref{eq:Def-MN},
and let $h$ be defined as~\eqref{eq:Def-h}.
Then, there exists a constant $C$ independent of $n$ such that
\[
 \sup_{t\in (-\infty,\infty)}\left|
 f^{(m)}(t)
- \sum_{k=-M}^N f(\psi_4(kh))
\left(\frac{\mathrm{d}}{\mathrm{d}t}\right)^m
\left\{
S(k,h)(\psi_4^{-1}(t))
\right\}
\right|
\leq C n^{(m+1)/2}\exp\left(-\sqrt{\pi d \mu n}\right).
\]
\end{theorem}

\section{Improvement of the convergence rate}
\label{sec:improve-converge}

In this section, we present the new results provided in this study.
We describe the convergence theorems of
our formula~\eqref{eq:new-formula}
for $i=2$ and $i=4$ in Section~\ref{subsec:our-result}.
The proof is provided in Section~\ref{sec:proofs}.

\subsection{Convergence theorems for our formula}
\label{subsec:our-result}

Convergence theorems for the new formula are described as follows.

\begin{theorem}
\label{thm:our-2}
Assume that $f$ is analytic in $\phi_2(\domD_d)$ with $0<d<\pi$,
and that there exist positive constants $K$, $\alpha$ and $\beta$
such that~\eqref{eq:f-g_2-bound}
holds for all $z\in\phi_2(\domD_d)$.
Let $\mu = \min\{\alpha,\beta\}$,
let $M$ and $N$ be defined as~\eqref{eq:Def-MN},
and let $h$ be defined as~\eqref{eq:Def-h}.
Then, there exists a constant $C$ independent of $n$ such that
\[
 \sup_{t\in (0,\infty)}\left|
 f^{(m)}(t)
- \sum_{k=-M}^N \frac{f(\phi_2(kh))}{g_2(\phi_2(kh))}
\left(\frac{\mathrm{d}}{\mathrm{d}t}\right)^m
\left\{
g_2(t)S(k,h)(\phi_2^{-1}(t))
\right\}
\right|
\leq C n^{(m+1)/2}\exp\left(-\sqrt{\pi d \mu n}\right).
\]
\end{theorem}
\begin{theorem}
\label{thm:our-4}
Assume that $f$ is analytic in $\phi_4(\domD_d)$ with $0<d<\pi$,
and that there exist positive constants $K$, $\alpha$ and $\beta$
such that
\begin{align}
 |f(z)|&\leq K |\E^{-z}|^{\beta}\label{leq:f-Dd-plus}
\end{align}
holds for all $z\in\phi_4(\domD_d^{+})$,
and~\eqref{leq:f-Dd-minus-original}
holds for all $z\in\phi_4(\domD_d^{-})$.
Let $\mu = \min\{\alpha,\beta\}$,
let $M$ and $N$ be defined as~\eqref{eq:Def-MN},
and let $h$ be defined as~\eqref{eq:Def-h}.
Then, there exists a constant $C$ independent of $n$ such that
\[
 \sup_{t\in (-\infty,\infty)}\left|
 f^{(m)}(t)
- \sum_{k=-M}^N f(\phi_4(kh))
\left(\frac{\mathrm{d}}{\mathrm{d}t}\right)^m
\left\{
S(k,h)(\phi_4^{-1}(t))
\right\}
\right|
\leq C n^{(m+1)/2}\exp\left(-\sqrt{\pi d \mu n}\right).
\]
\end{theorem}

The large difference between Stenger's and
our theorems is the upper bound of $d$;
$d<\pi/2$ in Theorems~\ref{thm:Stenger-2} and~\ref{thm:Stenger-4},
whereas $d<\pi$ in Theorems~\ref{thm:our-2} and~\ref{thm:our-4}.
This implies that the value of $d$ in our theorems
may be larger than that in Stenger's theorems.
Especially in Theorem~\ref{thm:our-4},
the value of $\mu$ may also be larger than that in Theorem~\ref{thm:Stenger-4},
because $\beta$ should be twice as large
in view of~\eqref{leq:f-Dd-plus-original} and~\eqref{leq:f-Dd-plus}.
The values of $d$ and $\mu$ cause a difference in the convergence rate,
which is estimated in common as
$\OO(n^{(m+1)/2}\exp(-\sqrt{\pi d \mu n}))$.

\subsection{Proofs}
\label{sec:proofs}

In this subsection, proofs of the new theorems stated in Section~\ref{subsec:our-result}
are provided.
First, we describe the organization of the proof
in Section~\ref{sec:sketch-proof}.
Then, we prove Theorem~\ref{thm:our-2} in Section~\ref{sec:our-2},
and Theorem~\ref{thm:our-4} in Section~\ref{sec:our-4}.

\subsubsection{Sketch of the proof}
\label{sec:sketch-proof}

To estimate the error of~\eqref{eq:new-formula},
we divide it into two terms as
\begin{align*}
&\left|
f^{(m)}(t) -
\sum_{k=-M}^N \frac{f(\phi_i(kh))}{g_i(\phi_i(kh))}
\left(\frac{\mathrm{d}}{\mathrm{d}t}\right)^m
\left\{
g_i(t)S(k,h)(\phi_i^{-1}(t))
\right\}
\right|\\
&\leq \left|
f^{(m)}(t) -
\sum_{k=-\infty}^{\infty} \frac{f(\phi_i(kh))}{g_i(\phi_i(kh))}
\left(\frac{\mathrm{d}}{\mathrm{d}t}\right)^m
\left\{
g_i(t)S(k,h)(\phi_i^{-1}(t))
\right\}
\right|\\
&\quad +
\left|
\sum_{k=-\infty}^{-M-1}
\frac{f(\phi_i(kh))}{g_i(\phi_i(kh))}
\left(\frac{\mathrm{d}}{\mathrm{d}t}\right)^m
\left\{
g_i(t)S(k,h)(\phi_i^{-1}(t))
\right\}
+\sum_{k=N+1}^{\infty}
\frac{f(\phi_i(kh))}{g_i(\phi_i(kh))}
\left(\frac{\mathrm{d}}{\mathrm{d}t}\right)^m
\left\{
g_i(t)S(k,h)(\phi_i^{-1}(t))
\right\}
\right|,
\end{align*}
which are referred to as the discretization error and truncation error,
respectively.
To analyze the discretization error,
the following function space is important.

\begin{definition}
Let $d$ be a positive constant, and
let $\domD_d(\epsilon)$ be a rectangular domain defined
for $0<\epsilon<1$ by
\[
\domD_d(\epsilon)
= \{\zeta\in\mathbb{C}:|\Re\zeta|<1/\epsilon,\, |\Im\zeta|<d(1-\epsilon)\}.
\]
Then, $\mathbf{H}^1(\domD_d)$ denotes the family of all analytic functions $F$
on $\domD_d$ such that the norm $\mathcal{N}_1(F,d)$ is finite, where
\[
\mathcal{N}_1(F,d)
=\lim_{\epsilon\to 0}\oint_{\partial \domD_d(\epsilon)} |F(\zeta)||\mathrm{d}\zeta|.
\]
\end{definition}

Under the definition,
the discretization error has been estimated as follows.

\begin{theorem}[Okayama and Tanaka~{\cite[Theorem~3]{okayama23:_error_sinc}}]
Let $d>0$, let $\phi$ be a conformal map that maps $\domD_d$ onto
a complex domain $\domD$ containing $(a,b)$,
and let $F(\zeta)=f(\phi(\zeta))/g(\phi(\zeta))$.
Assume that $F\in\mathbf{H}^1(\domD_d)$.
Moreover,
assume that there exists a positive constant $\tilde{C}_1$,
for all nonnegative integers $\{\lambda_l\}_{l=0}^{j}$ satisfying
$\lambda_0=0$ and $\sum_{l=1}^{j} l \lambda_l = j$, it holds that
\begin{equation}
\sup_{t\in(a, b)}
\left|
\left\{g(t)\sin\left(\frac{\pi\phi^{-1}(t)}{h}\right)\right\}^{(m-j)}
\prod_{l=0}^{j}\left[\left(\phi^{-1}(t)\right)^{(l)}\right]^{\lambda_l}
\right|
\leq \tilde{C}_1 h^{-m}
\quad (j=0,\,1,\,2,\,\ldots,\,m).
\label{eq:modified-Stenger-bound-in-disc}
\end{equation}
Then, there exists a constant $\tilde{C}_2$ independent of $h$ such that
\begin{align}
\sup_{t\in(a,b)}\left|
f^{(m)}(t)
- \left(\frac{\mathrm{d}}{\mathrm{d} t}\right)^m
\sum_{k=-\infty}^{\infty} \frac{f(\phi(kh))}{g(\phi(kh))}
\left(\frac{\mathrm{d}}{\mathrm{d}t}\right)^m
\left\{
g(t)S(k,h)(\phi^{-1}(t))
\right\}
\right|
&\leq \frac{\tilde{C}_2 \mathcal{N}_1(F,d)}{h^m \sinh(\pi d/h)}.
\label{eq:discretization-error}
\end{align}
\end{theorem}

The most difficult point
to use this theorem
is showing the condition~\eqref{eq:modified-Stenger-bound-in-disc}.
We slightly change the condition~\eqref{eq:modified-Stenger-bound-in-disc}
as follows,
which can be derived easily by using
$\sin\theta = (\E^{\I\theta}-\E^{-\I\theta})/(2\I)$.

\begin{theorem}
\label{thm:discretization-error}
Let $d>0$, let $\phi$ be a conformal map that maps $\domD_d$ onto
a complex domain $\domD$ containing $(a,b)$,
and let $F(\zeta)=f(\phi(\zeta))/g(\phi(\zeta))$.
Assume that $F\in\mathbf{H}^1(\domD_d)$.
Moreover,
assume that there exists a positive constant $\tilde{C}_1$,
for all nonnegative integers $\{\lambda_l\}_{l=0}^{j}$ satisfying
$\lambda_0=0$ and $\sum_{l=1}^{j} l \lambda_l = j$, it holds that
\begin{equation}
\sup_{\substack{t\in(a, b)\\ s\in [-\pi/h,\pi,h]}}
\left|
\left(\frac{\uppartial}{\uppartial t}\right)^{m-j}
\left\{g(t)\E^{\I s \phi^{-1}(t)}\right\}
\prod_{l=0}^{j}\left[\left(\phi^{-1}(t)\right)^{(l)}\right]^{\lambda_l}
\right|
\leq \tilde{C}_1 h^{-m}
\quad (j=0,\,1,\,2,\,\ldots,\,m).
\label{eq:modified-TK-bound-in-disc}
\end{equation}
Then, there exists a constant $\tilde{C}_2$ independent of $h$
such that~\eqref{eq:discretization-error} holds.
\end{theorem}

We use this theorem
to estimate the discretization error
(setting $g=g_i$ and $\phi=\phi_i$ ($i=2,\,4$)).
To estimate the truncation error,
we use the following lemma.

\begin{lemma}[Stenger~{\cite[Part of Theorem~4.4.2]{stenger93:_numer}}]
\label{lem:truncation-error}
Let $d>0$, let $\phi$ be a conformal map that maps $\domD_d$ onto
a complex domain $\domD$ containing $(a,b)$,
and let $F(\zeta)=f(\phi(\zeta))/g(\phi(\zeta))$.
Assume that there exist positive constants $R$, $\alpha$ and $\beta$
such that
\begin{equation}
 |F(x)|\leq
\frac{R}{(1 + \E^{-x})^{\alpha}(1 + \E^{x})^{\beta}}
\label{leq:F-exponential-decay-real}
\end{equation}
holds for all $x\in\mathbb{R}$.
Moreover, assume that there exists a constant $\tilde{C}_1$ such that
\begin{equation}
\sup_{\substack{t\in(a, b) \\ s\in [-\pi/h,\pi/h]}}
\left|
\left(\frac{\uppartial}{\uppartial t}\right)^{m}g(t)\E^{\I s\phi^{-1}(t)}
\right|
\leq \tilde{C}_1 h^{-m}.
\label{eq:Stenger-truncate-bound}
\end{equation}
Let $\mu = \min\{\alpha,\beta\}$,
and
let $M$ and $N$ be defined as~\eqref{eq:Def-MN}.
Then, it holds that
\[
\sup_{t\in(a,b)}
\left|
\sum_{k=-\infty}^{-M-1}
\frac{f(\phi(kh))}{g(\phi(kh))}
\left(\frac{\mathrm{d}}{\mathrm{d}t}\right)^m
\left\{
g(t)S(k,h)(\phi^{-1}(t))
\right\}
+\sum_{k=N+1}^{\infty}
\frac{f(\phi(kh))}{g(\phi(kh))}
\left(\frac{\mathrm{d}}{\mathrm{d}t}\right)^m
\left\{
g(t)S(k,h)(\phi^{-1}(t))
\right\}
\right|
\leq
 \tilde{C}_1 h^{-m} \frac{2R}{\mu h}\E^{-\mu n h}.
\]
\end{lemma}

Combining Theorem~\ref{thm:discretization-error}
and Lemma~\ref{lem:truncation-error},
setting $h$ as~\eqref{eq:Def-h},
we have
\begin{align*}
  \left|
 f^{(m)}(t)
- \sum_{k=-M}^N \frac{f(\phi_i(kh))}{g_i(\phi_i(kh))}
\left(\frac{\mathrm{d}}{\mathrm{d}t}\right)^m
\left\{
g_i(t)S(k,h)(\phi_i^{-1}(t))
\right\}
\right|
&\leq \frac{2\tilde{C}_2 \mathcal{N}_1(F,d)}{1 - \E^{-2\pi d/h}}
h \cdot h^{-m-1}\E^{-\pi d/h}
+\frac{2 \tilde{C}_1 R}{\mu} h^{-m-1}\E^{-\mu n h}\\
&= 2\left\{
\frac{\tilde{C}_2\mathcal{N}_1(F,d)}{1 - \E^{-2\sqrt{\pi d \mu n}}}
\sqrt{\frac{\pi d}{\mu n}}
+ \frac{\tilde{C}_1 R}{\mu}
\right\}
\left(\sqrt{\frac{\mu n}{\pi d}}\right)^{m+1}\E^{-\sqrt{\pi d \mu n}}\\
&\leq 2\left\{
\frac{\tilde{C}_2\mathcal{N}_1(F,d)}{1 - \E^{-2\sqrt{\pi d \mu}}}
\sqrt{\frac{\pi d}{\mu}}
+ \frac{\tilde{C}_1 R}{\mu}
\right\}
\left(\sqrt{\frac{\mu n}{\pi d}}\right)^{m+1}\E^{-\sqrt{\pi d \mu n}},
\end{align*}
from which we obtain the desired error bound.
This completes the proof of Theorems~\ref{thm:our-2}
and~\ref{thm:our-4}.

\subsubsection{Proof of Theorem~\ref{thm:our-2}}
\label{sec:our-2}

First, to use Theorem~\ref{thm:discretization-error},
we should show $F\in\mathbf{H}^1(\domD_d)$
and the inequality~\eqref{eq:modified-TK-bound-in-disc}
with $g=g_2$ and $\phi=\phi_2$.
The first task is completed by the following result.

\begin{lemma}[Okayama et al.~{\cite[Lemma~4.4]{OkaShinKatsu}}]
Assume that $\tilde{f}$ is analytic in $\phi_2(\domD_d)$ with $0<d<\pi$,
and that there exist positive constants $K$, $\alpha$ and $\beta$
such that
\[
 \left|\tilde{f}(z)\right|
\leq K \left|\frac{z}{1+z}\right|^{\alpha}|\E^{-z}|^{\beta}
\]
holds for $z\in\phi_2(\domD_d)$.
Then, putting $F(\zeta)=\tilde{f}(\phi_2(\zeta))$,
we have $F\in\mathbf{H}^1(\domD_d)$.
\end{lemma}

Setting $\tilde{f}(z)=f(z)/g_2(z)$ in this lemma, we obtain
$F\in\mathbf{H}^1(\domD_d)$.
For the second task, we prepare the following theorem,
two propositions and two lemmas.

\begin{theorem}[Fa\`{a} di Bruno's formula (cf. Johnson~{\cite{johnson02:_faa_brunos}})]
\label{thm:Faa-di-Bruno}
Assume that $f$ and $\tilde{f}$ are
$j$ times continuously differentiable.
Then, it holds that
\begin{align*}
 \left(\frac{\mathrm{d}}{\mathrm{d} t}\right)^{j}
f(\tilde{f}(t))
=\sum \frac{j!}{k_1!k_2!\cdots k_j!}
f^{(k_1+k_2+\cdots+k_j)}(\tilde{f}(t))
\prod_{l=1}^j \left(\frac{\tilde{f}^{(l)}(t)}{l!}\right)^{k_l},
\end{align*}
where the sum is over all different solutions in
nonnegative integers $k_1$, $k_2$, \ldots, $k_j$ of
\begin{equation}
 1\cdot k_1 + 2\cdot k_2 + \cdots + j\cdot k_j = j.
\label{eq:sum-of-j-kj}
\end{equation}
\end{theorem}

\begin{proposition}
\label{prop:bound-g2-deriv}
For any nonnegative integer $j$, there exists a positive constant $C_j$
depending only on $j$ such that
\[
 \left|g_2^{(j)}(t)\right| \leq C_j (1 - \E^{-t})^{m-j}
\]
holds for all $t\in (0,\infty)$.
\end{proposition}
\begin{proof}
The claim clearly holds for $j=0$.
For any positive integer $j$, there exist constants
$a_1,\,a_2,\,\ldots,a_j$ such that
\[
 g_2^{(j)}(t) = (a_1 \E^{-t} + a_2 \E^{-2t} + \cdots + a_j \E^{-jt})
(1 - \E^{-t})^{m - j},
\]
which can be shown by induction.
Thus, the claim also holds for $j\geq 1$.
\end{proof}

\begin{proposition}
\label{prop:bound-phi-inv-deriv}
For any positive integer $j$, there exists a positive constant $C_j$
depending only on $j$ such that
\[
 \left|\left\{\phi_2^{-1}(t)\right\}^{(j)}\right|
 \leq C_j (1 - \E^{-t})^{-j}
\]
holds for all $t\in (0,\infty)$.
\end{proposition}
\begin{proof}
The claim clearly holds for $j=1$.
For any positive integer $j\geq 2$, there exist constants
$a_2,\,a_3,\,\ldots,a_{j-2}$ such that
\[
 \left\{\phi_2^{-1}(t)\right\}^{(j)}
 = (-1)^{j-1}(\E^{-t} + a_2 \E^{-2t} + a_3 \E^{-3t} + \cdots
 + a_{j-2}\E^{-(j-2)t} +  \E^{-(j-1)t})
(1 - \E^{-t})^{- j},
\]
which can be shown by induction.
Thus, the claim also holds for $j\geq 2$.
\end{proof}

\begin{lemma}
\label{lem:bound-prod-phi-inv-deriv}
Let $j$ be a nonnegative integer,
and $\{\lambda_l\}_{l=0}^{j}$ be nonnegative integers
satisfying $\lambda_0=0$ and $\sum_{l=1}^{j} l \lambda_l = j$.
Then, there exists a positive constant $C_j$
depending only on $j$ such that
\[
\left|
\prod_{l=0}^j\left[\left\{\phi_2^{-1}(t)\right\}^{(l)}\right]^{\lambda_l}
\right|
\leq C_j (1 - \E^{-t})^{-j}
\]
holds for all $t\in (0, \infty)$.
\end{lemma}
\begin{proof}
From Proposition~\ref{prop:bound-phi-inv-deriv},
there exists a positive constant $C_{l}$
such that
\[
\left|
\left[\left\{\phi_2^{-1}(t)\right\}^{(l)}\right]^{\lambda_l}
\right|
\leq \left[ C_{l} (1 - \E^{-t})^{-l} \right]^{\lambda_l}
= C_{l}^{\lambda_l} (1 - \E^{-t})^{-l\lambda_l}
\]
holds for all $t\in (0,\infty)$.
Thus, using $\lambda_0=0$ and $\sum_{l=1}^{j} l \lambda_l = j$,
we have
\begin{align*}
\left|
\prod_{l=0}^j\left[\left\{\phi_2^{-1}(t)\right\}^{(l)}\right]^{\lambda_l}
\right|
\leq
C_{1}^{\lambda_1}C_{2}^{\lambda_2}\cdots C_{j}^{\lambda_j}
(1 - \E^{-t})^{-(1\lambda_1 + 2\lambda_2 + \cdots + l\lambda_l)}
=C_{1}^{\lambda_1}C_{2}^{\lambda_2}\cdots C_{j}^{\lambda_j}
(1 - \E^{-t})^{-j},
\end{align*}
which shows the claim.
\end{proof}

\begin{lemma}
\label{lem:bound-exp-phi-inv}
Let a positive real number $H$ be given, and let $h\in(0,H]$.
Then, there exists a positive constant $C_{j,H}$
depending only on $j$ and $H$ such that
\[
\sup_{-\pi/h\leq s\leq \pi/h}
\left|
\left\{
\E^{\I s \phi_2^{-1}(t)}
\right\}^{(j)}
\right|
\leq C_{j,H} h^{-j} (1 - \E^{-t})^{-j}
\quad (j=0,\,1,\,2,\,\ldots,\,m)
\]
holds for all $t\in(0,\infty)$.
\end{lemma}
\begin{proof}
We use Fa\`{a} di Bruno's formula (Theorem~\ref{thm:Faa-di-Bruno})
with $f(t)=\E^{\I s t}$ and
$\tilde{f}(t)= \phi_2^{-1}(t)$.
Put $K_j =k_1 + k_2 + \cdots + k_j$,
where $k_1,\,k_2,\,\ldots,\,k_j$ are nonnegative integers
satisfying~\eqref{eq:sum-of-j-kj}.
Then, because $|s|\leq \pi/h$, it holds that
\begin{equation}
 \left|f^{(K_j)}(\tilde{f}(t))\right|
=|\I s|^{K_j}\left|\E^{\I s \tilde{f}(t)}\right| = |s|^{K_j}
\leq \frac{\pi^{K_j}}{h^{K_j}}
=\frac{\pi^{K_j}h^{j - K_j}}{h^{j}}
\leq \frac{\pi^{K_j}H^{j - K_j}}{h^{j}},
\label{eq:f-K_j-bound}
\end{equation}
where
$K_j \leq 1\cdot k_1 + 2\cdot k_2 + \cdots + j\cdot k_j = j$
is used at the last inequality.
Furthermore, setting $k_0 = 0$,
using Lemma~\ref{lem:bound-prod-phi-inv-deriv}
and~\eqref{eq:sum-of-j-kj},
we have
\begin{align*}
\left|
\prod_{l=1}^j \left[\frac{\tilde{f}^{(l)}(t)}{l!}\right]^{k_l}
\right|
\leq \left(\prod_{l=1}^j\frac{1}{(l!)^{k_l}}\right)
\cdot
\prod_{l=1}^j
\left|\tilde{f}^{(l)}(t)\right|^{k_l}
= \left(\prod_{l=1}^j\frac{1}{(l!)^{k_l}}\right)
\cdot
\prod_{l=0}^j
\left|\left\{\phi_2^{-1}(t)\right\}^{(l)}\right|^{k_l}
\leq
\left(\prod_{l=1}^j\frac{1}{(l!)^{k_l}}\right)
\cdot C_j (1 - \E^{-t})^{-j}.
\end{align*}
Combining the above estimates, we obtain the claim.
\end{proof}

Using the results above,
we show~\eqref{eq:modified-TK-bound-in-disc} as follows.

\begin{lemma}
\label{lem:modified-TK-bound-in-disc}
Let a positive real number $H$ be given, and let $h\in(0,H]$.
Let $j$ be a nonnegative integer,
and $\{\lambda_l\}_{l=0}^{j}$ be nonnegative integers
satisfying $\lambda_0=0$ and $\sum_{l=1}^{j} l \lambda_l = j$.
Then, there exists a positive constant $\tilde{C}_1$
such that~\eqref{eq:modified-TK-bound-in-disc} holds with
$(a,b)=(0,\infty)$, $g=g_2$ and $\phi = \phi_2$.
\end{lemma}
\begin{proof}
Using Proposition~\ref{prop:bound-g2-deriv},
Lemma~\ref{lem:bound-exp-phi-inv}
and the Leibniz rule,
we have
\begin{align*}
\left|
\left(\frac{\uppartial}{\uppartial t}\right)^{m-j}
\left[g_2(t)\E^{\I s\phi_2^{-1}(t)}\right]
\right|
&=\left|
\sum_{k=0}^{m-j}\binom{m-j}{k} g_2^{(m-j-k)}(t)
\left(\E^{\I s \phi_2^{-1}(t)}\right)^{(k)}
\right|\\
&\leq \sum_{k=0}^{m-j}\binom{m-j}{k} C_{m-j-k} (1 - \E^{-t})^{m-(m-j-k)}
C_{k,H} h^{- k} (1 - \E^{-t})^{-k}\\
&= (1 - \E^{-t})^{j} h^{-m}
\sum_{k=0}^{m-j}\binom{m-j}{k} C_{m-j-k} C_{k,H} h^{m-k} \\
&\leq (1 - \E^{-t})^{j} h^{-m}
\sum_{k=0}^{m-j}\binom{m-j}{k} C_{m-j-k} C_{k,H} H^{m-k}.
\end{align*}
Combining the estimate with Lemma~\ref{lem:bound-prod-phi-inv-deriv},
we obtain the claim.
\end{proof}

Thus, we can use Theorem~\ref{thm:discretization-error}
for the discretization error.
Next, to use Lemma~\ref{lem:truncation-error}
for the truncation error,
we should show~\eqref{leq:F-exponential-decay-real}
(the inequality~\eqref{eq:Stenger-truncate-bound} clearly holds
from~\eqref{eq:modified-TK-bound-in-disc},
which is already shown by Lemma~\ref{lem:modified-TK-bound-in-disc}).
For the purpose, the following lemma is useful.

\begin{lemma}[Okayama et al.~{\cite[Lemma~4.7]{OkaShinKatsu}}]
\label{lem:bound-OkaShinKatsu}
It holds for all $x\in\mathbb{R}$ that
\[
 \left|
\frac{\log(1+\E^x)}{1 + \log(1 + \E^x)}
\right|
\leq \frac{1}{1+\E^{-x}}.
\]
\end{lemma}

Using this lemma, we show~\eqref{leq:F-exponential-decay-real}
as follows.

\begin{lemma}
Assume that there exist positive constants $K$, $\alpha$ and $\beta$
such that~\eqref{eq:f-g_2-bound}
holds for all $z\in (0,\infty)$.
Let $F(\zeta)=f(\phi_2(\zeta))/g_2(\phi_2(\zeta))$.
Then, there exists a constant $R$
such that~\eqref{leq:F-exponential-decay-real} holds
for all $x\in\mathbb{R}$.
\end{lemma}
\begin{proof}
From~\eqref{eq:f-g_2-bound} and Lemma~\ref{lem:bound-OkaShinKatsu},
it holds that
\[
|F(x)| =
 \left|
\frac{f(\phi_2(x))}{g_2(\phi_2(x))}
\right|
\leq K \left|\frac{\log(1+\E^x)}{1+\log(1+\E^x)}\right|^{\alpha}
\left|\E^{-\log(1+\E^x)}\right|^{\beta}
\leq K \frac{1}{(1 + \E^{-x})^{\alpha}}
\cdot\frac{1}{(1+\E^x)^{\beta}}.
\]
Hence,~\eqref{leq:F-exponential-decay-real} holds with $R = K$.
\end{proof}

Therefore, we can use Lemma~\ref{lem:truncation-error}
for the truncation error.
Thus, Theorem~\ref{thm:our-2} is established
by combining Theorem~\ref{thm:discretization-error}
and Lemma~\ref{lem:truncation-error}
as outlined in the sketch of the proof.

\subsubsection{Proof of Theorem~\ref{thm:our-4}}
\label{sec:our-4}

In the case of Theorem~\ref{thm:our-4} as well,
we use Theorem~\ref{thm:discretization-error}
and Lemma~\ref{lem:truncation-error}.
First, to use Theorem~\ref{thm:discretization-error},
we should show $F\in\mathbf{H}^1(\domD_d)$
and the inequality~\eqref{eq:modified-TK-bound-in-disc}
with $g=g_4$ and $\phi=\phi_4$.
The first task is completed by the following result.

\begin{lemma}[Okayama et al.~{\cite[Lemma~5.4]{tomoaki21:_new}}]
Assume that $\tilde{f}$ is analytic in $\phi_4(\domD_d)$ with $0<d<\pi$,
and that there exist positive constants $K$, $\alpha$ and $\beta$
such that~\eqref{leq:f-Dd-plus}
holds for all $z\in\phi_4(\domD_d^{+})$,
and~\eqref{leq:f-Dd-minus-original}
holds for all $z\in\phi_4(\domD_d^{-})$.
Then, putting $F(\zeta)=\tilde{f}(\phi_4(\zeta))$,
we have $F\in\mathbf{H}^1(\domD_d)$.
\end{lemma}

Setting $\tilde{f}(z)=f(z)/g_4(z)$ in this lemma, we obtain
$F\in\mathbf{H}^1(\domD_d)$.
For the second task, we prepare the following three propositions and one lemma.

\begin{proposition}
\label{prop:bound-p-deriv}
Let $p(t) = (t + \sqrt{4+t^2})/2$.
For any positive integer $l\geq 2$, there exists a positive constant $C_l$
depending only on $l$ such that
\begin{equation}
 \left|p^{(l)}(t)\right|\leq \frac{C_l}{(4 + t^2)^{(l+1)/2}}
\label{eq:bound-p-deriv}
\end{equation}
holds for all $t\in\mathbb{R}$.
\end{proposition}
\begin{proof}
For any positive integer $l\geq 2$, there exist constants
$a_0,\,a_1,\,\ldots,\,a_{l-2}$ such that
\[
 p^{(l)}(t)
 = \frac{a_0 + a_1 t + \cdots + a_{l-3} t^{l-3} + a_{l-2} t^{l-2}}{(4 + t^2)^{(2l - 1)/2}},
\]
which can be shown by induction.
Using $|t|=\sqrt{t^2} \leq \sqrt{4 + t^2}$, we have
\begin{align*}
\left|p^{(l)}(t)\right|
&\leq \frac{|a_0| + |a_1|(4 + t^2)^{1/2} + \cdots + |a_{l-3}|(4 + t^2)^{(l-3)/2} + |a_{l-2}|(4 + t^2)^{(l-2)/2}}{(4+t^2)^{(2l - 1)/2}}\\
&\leq \frac{|a_0|(4 + t^2)^{(l-2)/2} + |a_1|(4 + t^2)^{(l-2)/2} + \cdots + |a_{l-3}|(4 + t^2)^{(l-2)/2} + |a_{l-2}|(4 + t^2)^{(l-2)/2}}{(4+t^2)^{(2l - 1)/2}}\\
&=\frac{|a_0| + |a_1| + \cdots + |a_{l-3}| + |a_{l-2}|}{(4 + t^2)^{(l+1)/2}},
\end{align*}
which is the desired result.
\end{proof}

\begin{proposition}
\label{prop:sub-inequality}
It holds for all $t\in\mathbb{R}$ that
\[
 \frac{1}{1 - \E^{-(t + \sqrt{4+t^2})/2}}
\cdot \frac{1}{\sqrt{4+t^2}} \leq \frac{1}{2(1 - \E^{-1/2})}.
\]
\end{proposition}
\begin{proof}
Using $\sqrt{t^2} \geq -t$, we have
\[
 \frac{1}{2}\left(t + \sqrt{4 + t^2}\right)
= \frac{(t + \sqrt{4 + t^2})(\sqrt{4 + t^2} - t)}{2(\sqrt{4 + t^2} - t)}
\geq \frac{2}{\sqrt{4 + t^2} + \sqrt{t^2}}
\geq \frac{2}{\sqrt{4 + t^2} + \sqrt{4 + t^2}} = \frac{1}{\sqrt{4+t^2}},
\]
from which it holds that
\[
 \frac{1}{1 - \E^{-(t + \sqrt{4 + t^2})/2}}
\cdot \frac{1}{\sqrt{4 + t^2}}
\leq \frac{1}{1 - \E^{-1/\sqrt{4 + t^2}}}
\cdot \frac{1}{\sqrt{4 + t^2}}.
\]
Putting $u = 1/\sqrt{4+t^2}$ and $q(u) = u / (1 - \E^{-u})$,
we investigate the maximum of $q(u)$ for $u \in[0, 1/2]$.
Calculating the derivative of $q$ gives
\[
 q'(u) = \frac{r(u)}{(1 - \E^{-u})^2},
\]
where $r(u) = 1 - \E^{-u}(1 + u)$.
We readily have $r(u)\geq 0$ because $\E^{u}\geq 1 + u$ holds
for $u\in\mathbb{R}$.
Therefore, we have $q'(u)\geq 0$,
which implies that $q(u)$ monotonically increases for $u\in\mathbb{R}$.
Thus, $q(u)\leq q(1/2)$ holds for $u\in [0, 1/2]$,
which gives the desired inequality.
\end{proof}

\begin{proposition}
\label{prop:bound-phi-4-inv}
For any positive integer $j$, there exists a positive constant $C_j$
depending only on $j$ such that
\[
 \left|\left\{\phi_4^{-1}(t)\right\}^{(j)}\right|
 \leq C_j
\]
holds for all $t\in \mathbb{R}$.
\end{proposition}
\begin{proof}
Put $p(t) = (t + \sqrt{4+t^2})/2$.
We use Fa\`{a} di Bruno's formula (Theorem~\ref{thm:Faa-di-Bruno})
with $f(t)=\phi_2^{-1}(t)$ and
$\tilde{f}(t)= p(t)$
(note that $\phi_4^{-1}(t) = \phi_2^{-1}(p(t))$).
Put $K_j =k_1 + k_2 + \cdots + k_j$,
where $k_1,\,k_2,\,\ldots,\,k_j$ are nonnegative integers
satisfying~\eqref{eq:sum-of-j-kj}.
Then, from Proposition~\ref{prop:bound-phi-inv-deriv}, it holds that
\[
 \left|f^{(K_j)}(\tilde{f}(t))\right|
\leq C_{K_j} (1 - \E^{-p(t)})^{-K_j}
\leq C_{K_j} (1 - \E^{-p(t)})^{-j},
\]
where
$K_j \leq 1\cdot k_1 + 2\cdot k_2 + \cdots + j\cdot k_j = j$
is used at the last inequality.
Next, we consider the bound of $|\tilde{f}^{(l)}(t)|$
for $t< 0$ and $t\geq 0$ separately.
Note that if $t<0$, then~\eqref{eq:bound-p-deriv} holds for $l=1$ as well,
because
\[
 |\tilde{f}'(t)|
=|p'(t)|= \frac{\sqrt{4+t^2}+ t}{2\sqrt{4+t^2}}
\cdot\frac{\sqrt{4+t^2} - t}{\sqrt{4 + t^2} - t}
=\frac{2}{4 + t^2 - t \sqrt{4 + t^2}}
<\frac{2}{4 + t^2 - 0}.
\]
Therefore, for $t<0$, using~\eqref{eq:sum-of-j-kj}, we have
\begin{align*}
\left|f^{(K_j)}(t)\prod_{l=1}^{j}\left\{
\frac{\tilde{f}^{(l)}(t)}{l!}
\right\}^{k_l}\right|
&\leq C_{K_j} (1 - \E^{-p(t)})^{-j}
\prod_{l=1}^{j}\left\{
\frac{C_l}{l!(4 + t^2)^{(l+1)/2}}
\right\}^{k_l}\\
&\leq C_{K_j} (1 - \E^{-p(t)})^{-j}
\prod_{l=1}^{j}\left\{
\frac{C_l}{l!(4 + t^2)^{l/2}}
\right\}^{k_l}\\
&= \frac{C_{K_j}}{(1 - \E^{-p(t)})^{j}}
\left\{
\left(\frac{C_1}{1!}\right)^{k_1}
\left(\frac{C_2}{2!}\right)^{k_2}
\cdots
\left(\frac{C_j}{j!}\right)^{k_j}
\cdot\frac{1}{(4+t^2)^{(k_1 + 2k_2 + \cdots + j k_j)/2}}
\right\}\\
&= \frac{C_{K_j}}{(1 - \E^{-p(t)})^{j}}
\cdot \frac{1}{(4 + t^2)^{j/2}}
 \prod_{l=1}^{j}\left(\frac{C_l}{l!}\right)^{k_l}\\
&\leq \frac{C_{K_j}}{\left\{2 (1 - \E^{-1/2})\right\}^j}
\prod_{l=1}^{j}\left(\frac{C_l}{l!}\right)^{k_l},
\end{align*}
where Proposition~\ref{prop:sub-inequality} is used at
the last inequality.
Thus, the claim of this proposition follows for $t<0$.
Let $t\geq 0$ below.
For $l\geq 2$, from~\eqref{eq:bound-p-deriv}, we have
\[
 |\tilde{f}^{(l)}(t)|
\leq \frac{C_l}{(4 + t^2)^{(l+1)/2}}
\leq \frac{C_l}{(4 + 0)^{(l+1)/2}},
\]
and for $l=1$, we have
\[
 |\tilde{f}^{(l)}(t)|
=\frac{1}{2}\left(1 + \frac{t}{\sqrt{4 + t^2}}\right)
\leq\frac{1}{2}\left(1 + 1\right).
\]
Therefore, for any positive integer $l$,
$|\tilde{f}^{(l)}(t)|$ is bounded.
In addition, from Proposition~\ref{prop:bound-phi-inv-deriv}, it holds that
\[
 \left|f^{(K_j)}(\tilde{f}(t))\right|
\leq C_{K_j} (1 - \E^{-p(t)})^{-K_j}
\leq C_{K_j} (1 - \E^{-p(0)})^{-K_j}.
\]
Thus, the claim of this proposition follows for $t\geq 0$.
\end{proof}

\begin{lemma}
\label{lem:bound-exp-phi-inv-4}
Let a positive real number $H$ be given, and let $h\in(0,H]$.
Then, there exists a positive constant $C_{j,H}$
depending only on $j$ and $H$ such that
\[
\sup_{-\pi/h\leq s\leq \pi/h}
\left|
\left\{
\E^{\I s \phi_4^{-1}(t)}
\right\}^{(j)}
\right|
\leq C_{j,H} h^{-j}
\quad (j=0,\,1,\,2,\,\ldots,\,m)
\]
holds for all $t\in\mathbb{R}$.
\end{lemma}
\begin{proof}
We use Fa\`{a} di Bruno's formula (Theorem~\ref{thm:Faa-di-Bruno})
with $f(t)=\E^{\I s t}$ and
$\tilde{f}(t)= \phi_4^{-1}(t)$.
Put $K_j =k_1 + k_2 + \cdots + k_j$,
where $k_1,\,k_2,\,\ldots,\,k_j$ are nonnegative integers
satisfying~\eqref{eq:sum-of-j-kj}.
Then,~\eqref{eq:f-K_j-bound} holds
because $|s|\leq \pi/h$
and
$K_j \leq 1\cdot k_1 + 2\cdot k_2 + \cdots + j\cdot k_j = j$.
Furthermore, from Proposition~\ref{prop:bound-phi-4-inv},
we have
\[
 \prod_{l=1}^j
\left|\frac{\tilde{f}^{(l)}(t)}{l!}\right|^{k_l}
= \prod_{l=1}^j
\frac{1}{(l!)^{k_l}}\left|\left\{\phi_4^{-1}(t)\right\}^{(l)}\right|^{k_l}
\leq
\prod_{l=1}^j
\frac{1}{(l!)^{k_l}}C_l^{k_l}.
\]
Thus, the claim follows.
\end{proof}

Using the results above,
we show~\eqref{eq:modified-TK-bound-in-disc} as follows.

\begin{lemma}
\label{lem:modified-TK-bound-in-disc-4}
Let a positive real number $H$ be given, and let $h\in(0,H]$.
Let $j$ be a nonnegative integer,
and $\{\lambda_l\}_{l=0}^{j}$ be nonnegative integers
satisfying $\lambda_0=0$ and $\sum_{l=1}^{j} l \lambda_l = j$.
Then, there exists a positive constant $\tilde{C}_1$
such that~\eqref{eq:modified-TK-bound-in-disc} holds with
$(a,b)=(-\infty,\infty)$, $g=g_4$ and $\phi = \phi_4$.
\end{lemma}
\begin{proof}
Using Lemma~\ref{lem:bound-exp-phi-inv-4},
we have
\begin{align*}
\left|
\left(\frac{\uppartial}{\uppartial t}\right)^{m-j}
\left[g_4(t)\E^{\I s\phi_4^{-1}(t)}\right]
\right|
&=\left|
\left(\E^{\I s \phi_4^{-1}(t)}\right)^{(m-j)}
\right|\leq C_{m-j,H} h^{-(m-j)}
=C_{m-j,H} h^j h^{-m} \leq C_{m-j,H} H^j h^{-m}.
\end{align*}
Furthermore, from Proposition~\ref{prop:bound-phi-4-inv},
we have
\[
 \left|
\prod_{l=0}^j \left[\left\{\phi_4^{-1}(t)\right\}^{(l)}\right]^{\lambda_l}
\right|
\leq
\prod_{l=0}^j \left|\left\{\phi_4^{-1}(t)\right\}^{(l)}\right|^{\lambda_l}
\leq \prod_{l=0}^j C_l^{\lambda_l}.
\]
Combining these estimates,
we obtain the claim.
\end{proof}

Thus, we can use Theorem~\ref{thm:discretization-error}
for the discretization error.
Next, to use Lemma~\ref{lem:truncation-error}
for the truncation error,
we should show~\eqref{leq:F-exponential-decay-real}
(the inequality~\eqref{eq:Stenger-truncate-bound} clearly holds
from~\eqref{eq:modified-TK-bound-in-disc},
which is already shown by Lemma~\ref{lem:modified-TK-bound-in-disc-4}).
For the purpose, the following lemmas are useful.

\begin{lemma}[Okayama et al.~{\cite[Lemma~4.7]{tomoaki21:_new}}]
\label{lem:bound-Ddplus}
It holds for all $\zeta\in\overline{\domD_{\pi}^{+}}$ that
\[
 \left|\E^{1/\log(1 + \E^{\zeta})}\right|\leq \E^{1/\log 2}.
\]
\end{lemma}

\begin{lemma}[Okayama et al.~{\cite[Lemma~4.9]{tomoaki21:_new}}]
\label{lem:bound-Ddminus}
It holds for all $\zeta\in\overline{\domD_{\pi}^{-}}$ that
\[
 \left|\frac{1}{-1 + \log(1 + \E^{\zeta})}\right|\leq
\frac{1}{1 - \log 2}.
\]
\end{lemma}

Using these lemmas, we show~\eqref{leq:F-exponential-decay-real}
as follows.

\begin{lemma}
Assume that there exist positive constants $K$, $\alpha$ and $\beta$
such that~\eqref{leq:f-Dd-plus}
holds for all $z\in\phi_4(\domD_d^{+})$,
and~\eqref{leq:f-Dd-minus-original}
holds for all $z\in\phi_4(\domD_d^{-})$.
Let $F(\zeta)=f(\phi_4(\zeta))$.
Then, there exists a constant $R$
such that~\eqref{leq:F-exponential-decay-real} holds
for all $x\in\mathbb{R}$.
\end{lemma}
\begin{proof}
From~\eqref{leq:f-Dd-plus} and Lemma~\ref{lem:bound-Ddplus},
it holds for $x\geq 0$ that
\[
 |f(\phi_4(x))|
\leq K |\E^{-\phi_4(x)}|^{\beta}
= K |\E^{-\log(1+\E^{x})}|^{\beta} |\E^{1/\log(1+\E^{x})}|^{\beta}
\leq \frac{K \E^{\beta/\log 2}}{(1 + \E^x)^{\beta}}.
\]
Furthermore, because $1 + \E^{-x}\leq 1 + \E^{-0} = 2$
holds for $x\geq 0$, we have
\[
 \frac{K \E^{\beta/\log 2}}{(1 + \E^x)^{\beta}}
\leq  \frac{K \E^{\beta/\log 2}}{(1 + \E^x)^{\beta}}
\cdot \left(\frac{2}{1 + \E^{-x}}\right)^{\alpha}
= \frac{K \E^{\beta/\log 2} 2^{\alpha}}{(1+\E^{-x})^{\alpha}(1+\E^x)^{\beta}}.
\]
On the other hand,
from~\eqref{leq:f-Dd-minus-original} and
Lemmas~\ref{lem:bound-OkaShinKatsu} and~\ref{lem:bound-Ddminus},
it holds for $x< 0$ that
\[
 |f(\phi_4(x))|\leq K \frac{1}{|\phi_4(x)|^{\alpha}}
=K \left|\frac{\log(1+\E^x)}{1+\log(1+\E^x)}\right|^{\alpha}
\cdot\left|\frac{1}{-1 + \log(1+\E^x)}\right|^{\alpha}
\leq K \frac{1}{(1+\E^{-x})^{\alpha}}\cdot\frac{1}{(1 - \log 2)^{\alpha}}.
\]
Furthermore, because $1 + \E^{x}\leq 1 + \E^{0} = 2$
holds for $x < 0$, we have
\[
 K \frac{1}{(1+\E^{-x})^{\alpha}}\cdot\frac{1}{(1 - \log 2)^{\alpha}}
\leq  K \frac{1}{(1+\E^{-x})^{\alpha}}\cdot\frac{1}{(1 - \log 2)^{\alpha}}
\cdot \left(\frac{2}{1 + \E^x}\right)^{\beta}
= \frac{K 2^{\beta}}{(1 - \log 2)^{\alpha}}
\cdot \frac{1}{(1 + \E^{-x})^{\alpha}(1 + \E^x)^{\beta}}.
\]
Hence,~\eqref{leq:F-exponential-decay-real} holds with
\[
 R = \max\left\{
K \E^{\beta/\log 2} 2^{\alpha},
\frac{K 2^{\beta}}{(1 - \log 2)^{\alpha}}
\right\}.
\]
\end{proof}

Therefore, we can use Lemma~\ref{lem:truncation-error}
for the truncation error.
Thus, Theorem~\ref{thm:our-4} is established
by combining Theorem~\ref{thm:discretization-error}
and Lemma~\ref{lem:truncation-error}
as outlined in the sketch of the proof.

\section{Numerical examples}
\label{sec:numer}

This section presents numerical results.
All programs were written in the C programming language with double-precision floating-point
arithmetic.
In all examples, we set $m=2$ and approximated $f^{(l)}(t)$
for $l=0,\,1,\,2$.

First, we consider the following function
\begin{equation}
 f(t) = \sqrt{\frac{t}{1+t}}\E^{-t} (1 - \E^{-t})^2,\quad t\in(0,\infty),
\label{eq:example1}
\end{equation}
which is the case of $i=2$.
Therefore, we set $g_2(t) = (1 - \E^{-t})^2$.
The function $f$ satisfies the assumptions
of Theorem~\ref{thm:Stenger-2} with
$d=1.57$, $\alpha=1/2$ and $\beta=1$,
and also satisfies the assumptions
of Theorem~\ref{thm:our-2} with
$d=3.14$, $\alpha=1/2$ and $\beta=1$.
We investigated the errors on the following 101 points
\[
 t = t_i = 2^{i}, \quad i = -50,\,-49,\,\ldots,\,49,\,50,
\]
and maximum error among these points is plotted on the graph.
Only in the case of the second order derivative,
we used Mathematica with 20 digits of precision to compute $f''(t_i)$,
because the naive implementation in C did not give accurate results
(in contrast, approximate formulas were implemented purely in C
with double-precision).
The results are shown in Figs.~\ref{fig:example1_0}--\ref{fig:example1_2}.
We observe that in all cases,
the improved approximation formula~\eqref{eq:new-formula}
converges faster than Stenger's formula~\eqref{eq:Stenger-formula}.

\begin{figure}[htpb]
{\centering
\includegraphics[scale=.85]{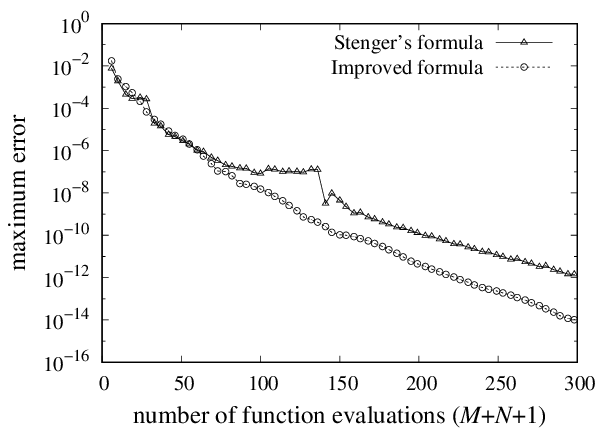}
\caption{Approximation errors of $f(t)$ in~\eqref{eq:example1}. $M$ and $N$ are defined by~\eqref{eq:Def-MN} with respect to $n$.}\label{fig:example1_0}
\includegraphics[scale=.85]{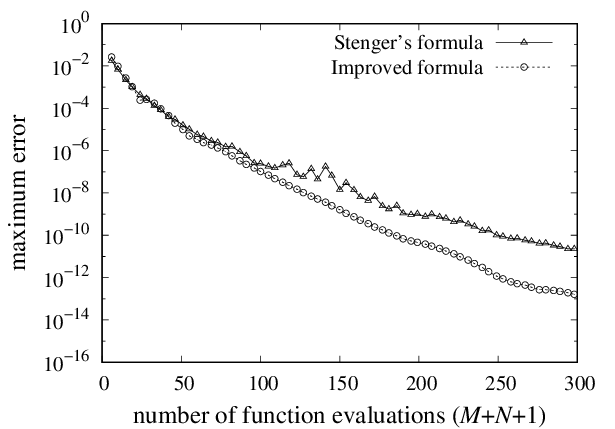}
\caption{Approximation errors of $f'(t)$ in~\eqref{eq:example1}. $M$ and $N$ are defined by~\eqref{eq:Def-MN} with respect to $n$.}\label{fig:example1_1}
\includegraphics[scale=.85]{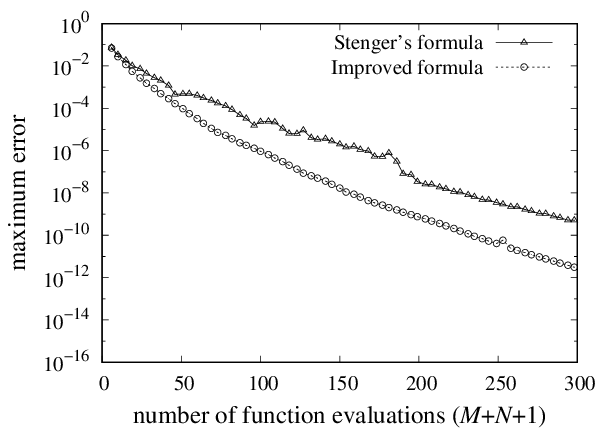}
\caption{Approximation errors of $f''(t)$ in~\eqref{eq:example1}. $M$ and $N$ are defined by~\eqref{eq:Def-MN} with respect to $n$.}\label{fig:example1_2}
}
\end{figure}

Next, we consider the following function
\begin{equation}
 f(t) = \frac{1}{(4 + t^2)(1 + \E^{\pi t/2})},\quad t\in(-\infty,\infty),
\label{eq:example2}
\end{equation}
which is the case of $i=4$.
Therefore, we set $g_4(t) = 1$.
The function $f$ satisfies the assumptions
of Theorem~\ref{thm:Stenger-4} with
$d=1.57$, $\alpha=2$ and $\beta=\pi/4$,
and also satisfies the assumptions
of Theorem~\ref{thm:our-4} with
$d=2.07$, $\alpha=2$ and $\beta=\pi/2$.
We investigated the errors on the following 202 points
\[
 t = \pm 2^{i}, \quad i = -50,\,-49,\,\ldots,\,49,\,50,
\]
and $t = 0$ (203 points in total), and
maximum error among these points is plotted on the graph.
The result is shown in Figs.~\ref{fig:example2_0}--\ref{fig:example2_2}.
We observe that in all cases,
the improved approximation formula~\eqref{eq:new-formula}
converges faster than Stenger's formula~\eqref{eq:Stenger-formula}.
Futhermore, the convergence profile of Stenger's formula appears relatively bumpy
compared to the improved approximation formula,
although the underlying mechanism remains unclear.

\begin{remark}
In all the experiments, not relative errors but absolute errors
were investigated, because
all convergence theorems
(Theorems~\ref{thm:Stenger-2}, \ref{thm:Stenger-4}, \ref{thm:our-2}
and~\ref{thm:our-4})
state the convergence on the absolute errors.
We note that the relative errors do not converge uniformly
on the given interval, in contrast to the absolute errors.
\end{remark}

\begin{figure}[htpb]
{\centering
\includegraphics[scale=.85]{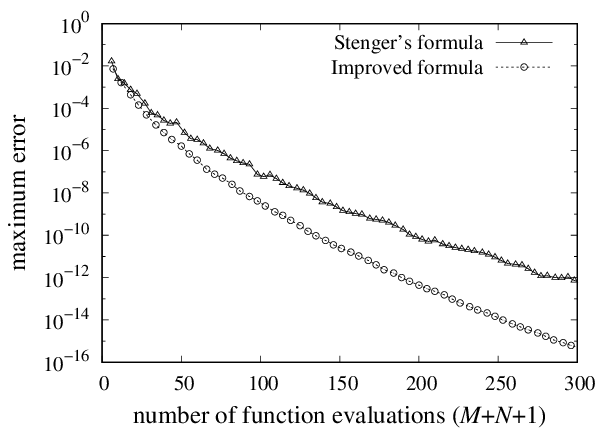}
\caption{Approximation errors of $f(t)$ in~\eqref{eq:example2}. $M$ and $N$ are defined by~\eqref{eq:Def-MN} with respect to $n$.}\label{fig:example2_0}
\includegraphics[scale=.85]{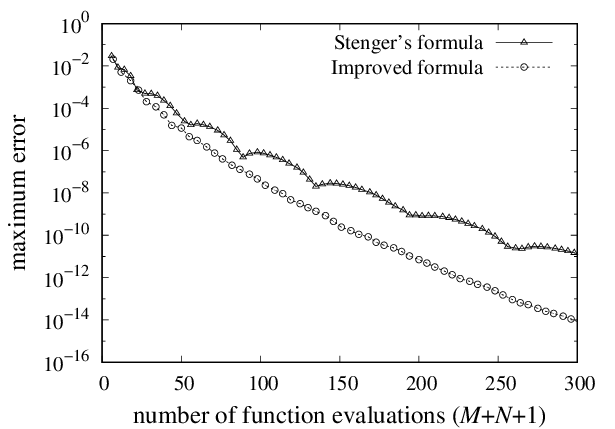}
\caption{Approximation errors of $f'(t)$ in~\eqref{eq:example2}. $M$ and $N$ are defined by~\eqref{eq:Def-MN} with respect to $n$.}\label{fig:example2_1}
\includegraphics[scale=.85]{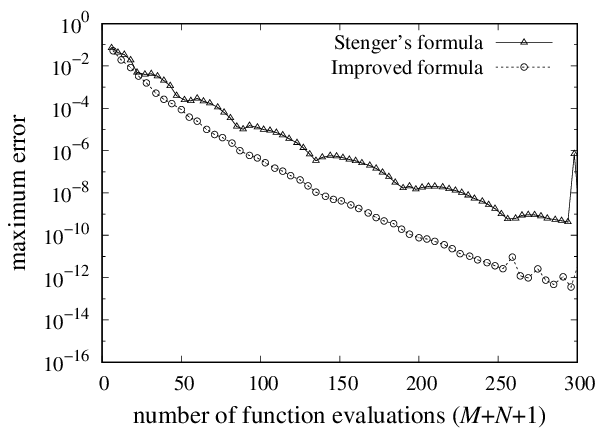}
\caption{Approximation errors of $f''(t)$ in~\eqref{eq:example2}. $M$ and $N$ are defined by~\eqref{eq:Def-MN} with respect to $n$.}\label{fig:example2_2}
}
\end{figure}

\section{Concluding remarks}
\label{sec:conclusion}

Stenger~\cite{stenger93:_numer,Stenger} proposed
approximation formulas for derivatives
based on the Sinc approximation~\eqref{eq:Sinc-approximation}
combined with appropriate conformal maps,
which were determined according to the target interval $(a,b)$
and the decay rate of the given function $f(t)$.
When $(a,b)=(0,\infty)$ and $|f(t)|$ decays exponentially as
$t\to\infty$, he adopted the conformal map $\psi_2(x)=\arsinh(\E^x)$.
When $(a,b)=(-\infty,\infty)$ and $|f(t)|$ decays algebraically as
$t\to-\infty$ and exponentially as $t\to\infty$,
he adopted the conformal map $\psi_4(x)=\sinh(\log(\arsinh(\E^x)))$.
He also provided convergence theorems of the two formulas
as described in Theorems~\ref{thm:Stenger-2} and~\ref{thm:Stenger-4},
which claim that the convergence rate of his formulas
is $\OO(n^{(m+1)/2}\exp(-\sqrt{\pi d \mu n}))$.
Here, $m$ denotes the order of derivative,
and $d$ and $\mu$ denote the parameters that are determined
by the regularity and the decay rate of $f(t)$, respectively.
In this study,
we proposed improved formulas by replacing the conformal maps
in his formula;
replace $\psi_2(x)$ with $\phi_2(x)=\log(1+\E^x)$,
and $\psi_4(x)$ with $\phi_4(x)=2\sinh(\log(\log(1+\E^x)))$.
Furthermore, we provided convergence theorems of the improved formulas
as described in Theorems~\ref{thm:our-2} and~\ref{thm:our-4},
which claim that the replacement may increase the value of $d$
and $\mu$ appearing in the convergence rate.
In fact, such improvements were observed in the numerical experiments
presented in Section~\ref{sec:numer}.

The conformal maps have a room for further improvement.
All conformal maps that appear in this paper
are categorized as the single-exponential transformations.
In various formulas based on the Sinc approximation,
acceleration of convergence has been reported
by replacing the single-exponential transformations with the
double-exponential transformations~\cite{Mori,murota25:_doubl,SugiharaMatsuo}.
Indeed, such an improvement was recently reported
for derivatives over the finite interval~\cite{okayama24:_approx_sinc_de}.
We are currently working on improvements for other cases
by employing the double-exponential transformations
adopted in different contexts~\cite{OkayamaSinc}.


\bibliography{SincApproxDeriv}

\begin{thebibliography}{20}
\expandafter\ifx\csname natexlab\endcsname\relax\def\natexlab#1{#1}\fi
\providecommand{\url}[1]{\texttt{#1}}
\providecommand{\href}[2]{#2}
\providecommand{\path}[1]{#1}
\providecommand{\DOIprefix}{doi:}
\providecommand{\ArXivprefix}{arXiv:}
\providecommand{\URLprefix}{URL: }
\providecommand{\Pubmedprefix}{pmid:}
\providecommand{\doi}[1]{\href{http://dx.doi.org/#1}{\path{#1}}}
\providecommand{\Pubmed}[1]{\href{pmid:#1}{\path{#1}}}
\providecommand{\bibinfo}[2]{#2}
\ifx\xfnm\relax \def\xfnm[#1]{\unskip,\space#1}\fi
\bibitem[{Bialecki(1991)}]{bialecki91:_sinc}
\bibinfo{author}{B.~Bialecki}, \bibinfo{title}{Sinc-collocation methods for
  two-point boundary value problems}, \bibinfo{journal}{IMA J.\ Numer.\ Anal.}
  \bibinfo{volume}{11} (\bibinfo{year}{1991}) \bibinfo{pages}{357--375}.
\bibitem[{Johnson(2002)}]{johnson02:_faa_brunos}
\bibinfo{author}{W.P. Johnson}, \bibinfo{title}{The curious history of {Fa\`{a}
  di Bruno's} formula}, \bibinfo{journal}{Amer.\ Math.\ Monthly}
  \bibinfo{volume}{109} (\bibinfo{year}{2002}).
\bibitem[{Lund and Bowers(1992)}]{lund92:_sinc_method_quadr_differ_equat}
\bibinfo{author}{J.~Lund}, \bibinfo{author}{K.L. Bowers}, \bibinfo{title}{Sinc
  Methods for Quadrature and Differential Equations},
  \bibinfo{publisher}{SIAM}, \bibinfo{address}{Philadelphia, PA},
  \bibinfo{year}{1992}.
\bibitem[{Mori(2005)}]{Mori}
\bibinfo{author}{M.~Mori}, \bibinfo{title}{Discovery of the double exponential
  transformation and its developments}, \bibinfo{journal}{Publ. RIMS, {\rm
  Kyoto Univ.}} \bibinfo{volume}{41} (\bibinfo{year}{2005})
  \bibinfo{pages}{897--935}.
\bibitem[{Morlet(1995)}]{morlet95:_conver}
\bibinfo{author}{A.C. Morlet}, \bibinfo{title}{Convergence of the sinc method
  for a fourth-order ordinary differential equation with an application},
  \bibinfo{journal}{SIAM J.\ Numer.\ Anal.} \bibinfo{volume}{32}
  (\bibinfo{year}{1995}) \bibinfo{pages}{1475--1503}.
\bibitem[{Murota and Matsuo(2025)}]{murota25:_doubl}
\bibinfo{author}{K.~Murota}, \bibinfo{author}{T.~Matsuo},
  \bibinfo{title}{Double exponential transformation: a quick review of a
  {Japanese} tradition}, \bibinfo{journal}{Jpn.\ J. Ind.\ Appl.\ Math.}
  \bibinfo{volume}{42} (\bibinfo{year}{2025}) \bibinfo{pages}{885--895}.
\bibitem[{Okayama(2018)}]{OkayamaSinc}
\bibinfo{author}{T.~Okayama}, \bibinfo{title}{Error estimates with explicit
  constants for the {Sinc} approximation over infinite intervals},
  \bibinfo{journal}{Appl.\ Math.\ Comput.} \bibinfo{volume}{319}
  (\bibinfo{year}{2018}) \bibinfo{pages}{125--137}.
\bibitem[{Okayama and Hirohata(2022)}]{okayama22:_theor}
\bibinfo{author}{T.~Okayama}, \bibinfo{author}{K.~Hirohata},
  \bibinfo{title}{Theoretical comparison of two conformal maps combined with
  the trapezoidal formula for the semi-infinite integral of exponentially
  decaying functions}, \bibinfo{journal}{JSIAM Lett.} \bibinfo{volume}{14}
  (\bibinfo{year}{2022}) \bibinfo{pages}{77--79}.
\bibitem[{Okayama and Kosaka(2024)}]{okayama24:_approx_sinc_de}
\bibinfo{author}{T.~Okayama}, \bibinfo{author}{T.~Kosaka},
  \bibinfo{title}{Approximation of derivatives over the finite interval via the
  {Sinc} approximation combined with the {DE} transformation and its
  theoretical error analysis}, \bibinfo{journal}{JSIAM Lett.}
  \bibinfo{volume}{16} (\bibinfo{year}{2024}) \bibinfo{pages}{77--80}.
\bibitem[{Okayama and Machida(2017)}]{OkayamaMachida}
\bibinfo{author}{T.~Okayama}, \bibinfo{author}{K.~Machida},
  \bibinfo{title}{Error estimate with explicit constants for the trapezoidal
  formula combined with {Muhammad-Mori's} {SE} transformation for the
  semi-infinite interval}, \bibinfo{journal}{JSIAM Lett.} \bibinfo{volume}{9}
  (\bibinfo{year}{2017}) \bibinfo{pages}{45--47}.
\bibitem[{Okayama et~al.(2021)Okayama, Nomura and Tsuruta}]{tomoaki21:_new}
\bibinfo{author}{T.~Okayama}, \bibinfo{author}{T.~Nomura},
  \bibinfo{author}{S.~Tsuruta}, \bibinfo{title}{New conformal map for the
  trapezoidal formula for infinite integrals of unilateral rapidly decreasing
  functions}, \bibinfo{journal}{J. Comput.\ Appl.\ Math.} \bibinfo{volume}{389}
  (\bibinfo{year}{2021}).
\bibitem[{Okayama et~al.(2020)Okayama, Shintaku and Katsuura}]{OkaShinKatsu}
\bibinfo{author}{T.~Okayama}, \bibinfo{author}{Y.~Shintaku},
  \bibinfo{author}{E.~Katsuura}, \bibinfo{title}{New conformal map for the
  {Sinc} approximation for exponentially decaying functions over the
  semi-infinite interval}, \bibinfo{journal}{J. Comput.\ Appl.\ Math.}
  \bibinfo{volume}{373} (\bibinfo{year}{2020}) \bibinfo{pages}{112358}.
\bibitem[{Okayama and Shiraishi(2021)}]{okayama21:_improv_sinc}
\bibinfo{author}{T.~Okayama}, \bibinfo{author}{T.~Shiraishi},
  \bibinfo{title}{Improvement of the conformal map combined with the {Sinc}
  approximation for unilateral rapidly decreasing functions},
  \bibinfo{journal}{JSIAM Lett.} \bibinfo{volume}{13} (\bibinfo{year}{2021})
  \bibinfo{pages}{37--39}.
\bibitem[{Okayama and Tanaka(2023)}]{okayama23:_error_sinc}
\bibinfo{author}{T.~Okayama}, \bibinfo{author}{K.~Tanaka},
  \bibinfo{title}{Error analysis of approximation of derivatives by means of
  the {Sinc} approximation for double-exponentially decaying functions},
  \bibinfo{journal}{JSIAM Lett.} \bibinfo{volume}{15} (\bibinfo{year}{2023})
  \bibinfo{pages}{5--8}.
\bibitem[{Saadatmandi and Razzaghi(2007)}]{saadatmandi07}
\bibinfo{author}{A.~Saadatmandi}, \bibinfo{author}{M.~Razzaghi},
  \bibinfo{title}{The numerical solution of third-order boundary value problems
  using sinc-collocation method}, \bibinfo{journal}{Comm.\ Numer.\ Methods
  Engrg.} \bibinfo{volume}{23} (\bibinfo{year}{2007})
  \bibinfo{pages}{681--689}.
\bibitem[{Stenger(1978)}]{stenger78:_optim_h}
\bibinfo{author}{F.~Stenger}, \bibinfo{title}{Optimal convergence of minimum
  norm approximations in {$H_p$}}, \bibinfo{journal}{Numer.\ Math.}
  \bibinfo{volume}{29} (\bibinfo{year}{1978}) \bibinfo{pages}{345--362}.
\bibitem[{Stenger(1993)}]{stenger93:_numer}
\bibinfo{author}{F.~Stenger}, \bibinfo{title}{Numerical Methods Based on Sinc
  and Analytic Functions}, \bibinfo{publisher}{Springer-Verlag},
  \bibinfo{address}{New York}, \bibinfo{year}{1993}.
\bibitem[{Stenger(2011)}]{Stenger}
\bibinfo{author}{F.~Stenger}, \bibinfo{title}{Handbook of Sinc Numerical
  Methods}, \bibinfo{publisher}{CRC Press}, \bibinfo{address}{Boca Raton, FL},
  \bibinfo{year}{2011}.
\bibitem[{Sugihara(2003)}]{sugihara03:_near}
\bibinfo{author}{M.~Sugihara}, \bibinfo{title}{Near optimality of the sinc
  approximation}, \bibinfo{journal}{Math.\ Comput.} \bibinfo{volume}{72}
  (\bibinfo{year}{2003}) \bibinfo{pages}{767--786}.
\bibitem[{Sugihara and Matsuo(2004)}]{SugiharaMatsuo}
\bibinfo{author}{M.~Sugihara}, \bibinfo{author}{T.~Matsuo},
  \bibinfo{title}{Recent developments of the {Sinc} numerical methods},
  \bibinfo{journal}{J. Comput.\ Appl.\ Math.} \bibinfo{volume}{164--165}
  (\bibinfo{year}{2004}) \bibinfo{pages}{673--689}.

\end{thebibliography}

\end{document}